\documentclass[12pt]{amsart}
\usepackage{ifthen,verbatim}
\usepackage{graphicx}
\usepackage{mathrsfs}
\usepackage{amsmath,amssymb,amsthm}
\usepackage{tikz-cd}
\usepackage{color}
\usepackage{caption}
\usepackage{subcaption}
\numberwithin{equation}{section}
\nonstopmode
\usepackage{amsfonts}
\usepackage{amssymb}
\usepackage[english]{babel}
\usepackage[utf8x]{inputenc}
\usepackage[autostyle]{csquotes}

\usepackage{float}
\usepackage{tabularx}
\newcolumntype{C}{>{\centering\arraybackslash}X}

\usepackage{tikz-cd}
\usepackage{color}
\usepackage{caption}

\newtheorem{theorem}{Theorem}[section]
\newtheorem{corollary}{Corollary}[theorem]
\newtheorem{lemma}[theorem]{Lemma}

\newenvironment{myproof}[2] {\paragraph{\it Proof of {#1} {#2}.}}{\hfill$\square$}
\usepackage{url}
\usepackage{graphicx}
\usepackage{multirow}
\usepackage{tabularx}
\usepackage{longtable}
\usepackage{enumerate}

\usepackage{tikz-cd}

\usepackage{bm}
\usepackage{amsmath,mathtools}
\usepackage{amsfonts}

\newcommand{\RomanNumeralCaps}[1]
    {\MakeUppercase{\romannumeral #1}}

\renewcommand{\Im}{{\operatorname{Im}\,}}
\usepackage[skip=1pt,font=scriptsize]{caption}

\newtheorem{innercustomthm}{Theorem}
\newenvironment{customthm}[1]
  {\renewcommand\theinnercustomthm{#1}\innercustomthm}
  {\endinnercustomthm}

\usepackage{dirtytalk}

\usepackage{mathtools}

\begin{document}
\title
{First Derivative of Automorphic Function of Triangle Group}

\makeatletter\def\thefootnote{\@arabic\c@footnote}\makeatother

\author[Md. S. Alam]{Md. Shafiul Alam}
\address{
Department of Mathematics, University of Barishal, Barishal-8254, Bangladesh}
\email{msalam@bu.ac.bd, shafiulmt@gmail.com}
\author[B. K. Saha]{Bijan Krishna Saha}
\address{Department of Mathematics, University of Barishal, Barishal-8254, Bangladesh} 
\email{bijandumath@gmail.com}

\author[C. Podder]{Chinmayee Podder}
\address{Department of Mathematics, University of Barishal, Barishal-8254, Bangladesh} 
\email{chinmayeepodder1@gmail.com}

\begin{abstract}
For a triangle group $G$, the $G$-automorphic function is the inverse of Schwarz triangle function. In this paper, we compute the first derivative of the $G$-automorphic function for the triangle group $G$ in terms of the Gaussian hypergeometric function.  
\end{abstract}

\maketitle

\section{Introduction}

If $\alpha, \beta, \gamma\in\mathbb{C}$ for $\gamma\neq 0,\,-1,\,-2,\,\dots$,  and $j\in\mathbb{N}\cup\{0\}$, then the function $_2F_1(\alpha, \beta; \gamma;\xi)$ defined by
\begin{align*}
    _2F_1(\alpha, \beta; \gamma;\xi)=\sum_{n=0}^{\infty}\frac{(\alpha, j) (\beta, j)}{(\gamma, j) j!}\xi^j,\quad \lvert \xi \rvert < 1,
\end{align*}
where 
\begin{equation*}
    (\alpha, j)=
    \begin{cases}
        1 & \text{if } j=0 \text{ and } \alpha \neq 0 \\
        \alpha(\alpha+1)\cdots(\alpha+j-1) & \text{if } j\geq1
    \end{cases},
\end{equation*}
is known as the Gaussian hypergeometric function.
By the Euler's integral representation formula, one can extend $_2F_1(\alpha, \beta; \gamma;\xi)$ analytically to $\mathbb{C}\setminus [1,+\infty)$ (cf. \cite[Chapter \RomanNumeralCaps{2}] {bateman}, \cite[Chapter \RomanNumeralCaps{14}]{whit}).

For the group $\operatorname{SL_2}(\mathbb{R})$ given by 
$$ 
\operatorname{SL_2}(\mathbb{R})=\Bigg\{\displaystyle\begin{pmatrix}  a & b\\ c & d \end{pmatrix}: a, b, c, d \in\mathbb{R},~\ ad-bc=1\Bigg\},
$$
the group $\operatorname{PSL_2}(\mathbb{R})$ is defined as $\operatorname{PSL_2}(\mathbb{R})=\operatorname{SL_2}(\mathbb{R})/\{I_2, -I_2\}$, where $I_2$ is the $2\times2$ identity matrix (see \cite[Chapter \RomanNumeralCaps{7}]{serre}). Let $\mathcal{H}$ denote the complex upper half-plane $\{w\in\mathbb{C}: \Im w>0\}$. The action of $\operatorname{PSL_2}(\mathbb{R})$ on $\mathcal{H}$ is given by   
  \[w\mapsto \begin{pmatrix}
   a & b\\ c & d
   \end{pmatrix}w= \frac{a w+b}{c w+d},\, \text{for}\,\begin{pmatrix}
   a & b\\ c & d
   \end{pmatrix}\in \operatorname{PSL_2}(\mathbb{R}), \, w\in\mathcal{H}.\]

Let $A=\begin{pmatrix}
   a & b\\ c & d
   \end{pmatrix}\in \operatorname{PSL_2}(\mathbb{R})$. The element $A$ is called parabolic if $a+d=\pm 2$, elliptic if $(a+d)\in\mathbb{R}$ and $|a+d|<2$, hyperbolic if $(a+d)\in\mathbb{R}$ and $|a+d|>2$, loxodromic if $(a+d)\in\mathbb{C}$ (see Theorem 15, Section 10 of \cite{ford}). Let $G\subset\operatorname{PSL_2}(\mathbb{R})$ be a triangle group with signature $(0; n_1, n_2, n_3)$ acting on $\mathcal{H}$. Consider a parabolic element $A_1\in G$ and an elliptic element $A_2\in G$. If $A_1(x)=x$, where $x\in\mathbb{R}\cup\{\infty\}$, then the point $x$ is called a cusp of the group $G$. If $A_2(w)=w$, where $w\in\mathcal{H}$, then we call the point $w$ an elliptic point of $G$. Since the action of $G$ on $\mathcal{H}$ is properly discontinuous, we obtain the quotient Riemann surface $G\backslash\mathcal{H}$ (see \cite{beardon} and \cite{Katok:fg}). In \cite{alam2} and \cite{alam1}, a special type of triangle group known as the Hecke group is investigated to study modular equations.

   For a hyperbolic triangle $[w_1, w_2, w_3]$ with angles $\frac{\pi}{n_1},\,\frac{\pi}{n_2},\,\frac{\pi}{n_3}$, let $f$ denote the Schwarz triangle function. The $G$-automorphic function $\xi:\mathcal{H}\rightarrow{G\backslash\mathcal{H}^*}$, where $\mathcal{H}^*$ is the union of $\mathcal{H}$ and the set of cusps of $G$, is the inverse of $f$ and
   \[\xi\Big(\displaystyle\frac{aw+b}{cw+d}\Big)=\xi(w)\]
   for $w\in\mathcal{H}$ and $\begin{pmatrix} a & b\\ c & d
   \end{pmatrix}\in G$. The function $\xi(w)$ can be normalized in such a way that the orders of the ramifications are $n_1, n_2, n_3$ over $0,\,1,\,\infty$, respectively (see \cite{cohen}). Suppose that $n_1,\,n_2$ and $n_3$ are the orders of the elliptic points $w_1,\,w_2$ and $w_3$, respectively, such that
\begin{align*}
    \xi(w_1)=0,\quad \xi(w_2)=1,\quad \xi(w_3)=\infty.
\end{align*}
Then, $\xi(w)$ satisfies the following hypergeometric differential equation
\begin{equation}\label{hyp}
    \xi(1-\xi)\frac{d^2g}{d\xi^2}+\{\gamma-(\alpha+\beta+1)\xi\}\frac{dg}{d\xi}-\alpha\beta g=0.
\end{equation}

Let $g_1(\xi)$ and $g_2(\xi)$ be linearly independent solutions of (\ref{hyp}), then $f(\xi)=\displaystyle \frac{g_2(\xi)}{g_1(\xi)}$ is the Schwarz triangle function. The Schwarzian derivative, denoted by $\{f,\xi\}$, of $f(\xi)$ is given by \[\{f,\xi\}=\displaystyle\Big(\frac{f''(\xi)}{f'(\xi)}\Big)'-\frac{1}{2}\Big(\frac{f''(\xi)}{f'(\xi)}\Big)^2.\]
The function $f(z)$ satisfies the Schwarzian equation given by

\begin{equation}
    \{f,\xi\}=\frac{1-p_1^2}{2\xi^2}+\frac{1-p_2^2}{2(1-\xi)^2}+\frac{1-p_1^2-p_2^2+p_3^2}{2\xi(1-\xi)},
\end{equation}
where $p_1, p_2, p_3$ are known as accessory parameters and $p_1=\frac{1}{n_1},\, p_2=\frac{1}{n_2},\, p_3=\frac{1}{n_3}$ (cf. Chap. V, Sect. 7 of \cite{nehari}). These facts are presented in the following theorem.

\begin{customthm}{A}[$\text{\cite[Theorem 6.2]{bayer}}$]
For $\gamma \ne 1$, let \[\alpha'=\alpha-\gamma+1,~\ \beta'=\beta-\gamma+1, ~\ \gamma'=2-\gamma,\] then the functions 
$$g_1(\xi)={_2F_1}(\alpha, \beta; \gamma;\xi)\quad\text{and}\quad  g_2(\xi)=\xi^{1-\gamma} {_2F_1(\alpha', \beta'; \gamma';\xi)}$$ 
are linearly independent solutions of (\ref{hyp}). Let $w=f(\xi)=\displaystyle\frac{g_2(\xi)}{g_1(\xi)}$, then $\mathcal{H}=\{\xi\in\mathbb{C}: \Im \xi>0\}$ is mapped by $f(\xi)$ conformally to the hyperbolic (non-Euclidean) triangle $[w_1, w_2, w_3]$ in the $w$-plane, where 
\begin{align*}
    &w_1=f(0)=0,\\
    &w_2=f(1)=\displaystyle\frac{\Gamma(2-\gamma)\Gamma(\gamma-\alpha)\Gamma(\gamma-\beta)}{\Gamma(\gamma)\Gamma(1-\alpha)\Gamma(1-\beta)},\\
    &w_3=f(\infty)=\displaystyle e^{\pi i(1-\gamma)} \frac{\Gamma(\alpha)\Gamma(\gamma-\beta)\Gamma(2-\beta)}{\Gamma(\gamma)\Gamma(\alpha-\gamma+1)\Gamma(1-\beta)}.
\end{align*}
The interior angles at the vertices $w_1$, $w_2$ and $w_3$ are, respectively, $(1-\gamma)\pi,\,(\gamma-\alpha-\beta)\pi$ and $(\beta-\alpha)\pi$. 
\end{customthm}

For $w_1, w_2\in\mathcal{H}$, let us denote the hyperbolic (non-Euclidean) distance between $w_1$ and $w_2$ by $d_h(w_1, w_2)$. Since at the vertices $w_1, w_2$ and $w_3$ of the hyperbolic triangle $[w_1, w_2, w_3]$, the angles are $\displaystyle\frac{\pi}{n_1}$, $\displaystyle\frac{\pi}{n_2}$ and $\displaystyle\frac{\pi}{n_3}$, respectively, we have

\begin{equation}{\label{tanh1}}
    \tanh^2{\frac{d_h(w_1,w_2)}{2}}=\displaystyle\frac{\cos{\big\{\frac{1}{2}}\big(\frac{1}{n_1}+\frac{1}{n_2}+\frac{1}{n_3}\big)\pi\big\}\cos{\big\{\frac{1}{2}}\big(\frac{1}{n_1}+\frac{1}{n_2}-\frac{1}{n_3}\big)\pi\big\}}{\cos{\big\{\frac{1}{2}}\big(\frac{1}{n_2}+\frac{1}{n_3}-\frac{1}{n_1}\big)\pi\big\}\cos{\big\{\frac{1}{2}}\big(\frac{1}{n_3}+\frac{1}{n_1}-\frac{1}{n_2}\big)\pi\big\}}
\end{equation}
(see Equation (75.1) of \cite{carath}, where $a_i=d_h(w_1, w_2)$, $\alpha_i=\frac{\pi}{n_3}$, $\alpha_j=\frac{\pi}{n_1}$, $\alpha_k=\frac{\pi}{n_2}$). Also, by Equation (85.3) of \cite{carath}, we have
\begin{equation}\label{tanh2}
    \tanh^2{\frac{d_h(w_1,w_2)}{2}}=\displaystyle \Big|\frac{w_2-w_1}{w_2-\overline{w}_1}\Big|.
\end{equation}

In the following section, we present the main results and their proofs.
\\

\section{Main Results and Proofs}

\vspace{3mm}
\begin{theorem}\label{thm1}
Suppose that $G$ is a triangle group with signature $(0;n_1,n_2,n_3)$. Let $$\xi:\mathcal{H}\rightarrow{G\backslash\mathcal{H}^*}$$ be the $G$-automorphic function and let $w_1, w_2, w_3\in\mathcal{H}$ be the elliptic points of orders $n_1,\,n_2$ and $n_3$, respectively, such that 
\begin{align*}
    \xi(w_1)=0,\quad \xi(w_2)=1,\quad \xi(w_3)=\infty.
\end{align*}
Then,
\begin{equation}\label{c&k}
    \displaystyle \frac{d\xi}{dw}=C \xi^{(1-\frac{1}{n_1})}(1-\xi)^{(1-\frac{1}{n_2})}\Big({_2F_1}(\alpha, \beta; \gamma;\xi)-K \xi^{\frac{1}{n_1}}\ _2F_1(\alpha', \beta'; \gamma';\xi)\Big)^2,
\end{equation}
where $C\in \mathbb{C}\setminus \{0\}$, $\alpha, \beta\notin\mathbb{Z}$, $\alpha'=\alpha-\gamma+1, \beta'=\beta-\gamma+1, \gamma'=2-\gamma$ and 
\begin{align*}
    K=\bigg(\displaystyle\frac{\Gamma(1-\alpha)\Gamma(1-\beta)\Gamma(\alpha')\Gamma(\beta')}{\Gamma(\alpha)\Gamma(\beta)\Gamma(1-\alpha')\Gamma(1-\beta')}\bigg)^\frac{1}{2}\frac{\Gamma(\gamma)}{\Gamma(\gamma')}.
\end{align*}

\end{theorem}
\vspace{5mm}

We need the following lemma in the proof of Theorem \ref{thm1}.
\begin{lemma}\label{lem1}
Let $w_1, w_2\in\mathcal{H}$ be the elliptic points of orders $n_1$ and $n_2$, respectively. Then the hyperbolic (non-Euclidean) distance, $d_h(w_1, w_2)$, between $w_1$ and $w_2$ is given by

\begin{equation*}
    \tanh^2{\frac{d_h(w_1,w_2)}{2}}=\displaystyle\frac{\sin{\pi \alpha}\sin{\pi \beta}}{\sin{\pi \alpha'}\sin{\pi \beta'}}.
\end{equation*}
\end{lemma}

\begin{proof}
We have
\begin{equation}\label{n-abc}
    \frac{1}{n_1}=1-\gamma,\quad \frac{1}{n_2}=\gamma-\alpha-\beta,\quad \frac{1}{n_3}=\beta-\alpha
\end{equation}
  \[\]
and 
\begin{equation}\label{a'a}
    \alpha'=\alpha-\gamma+1,\quad \beta'=\beta-\gamma+1,\quad \gamma'=2-\gamma.
\end{equation}
From (\ref{n-abc}) and (\ref{a'a}), it follows that
\begin{equation}\label{abc}
    \alpha=\frac{1}{2}\Big(1-\frac{1}{n_1}-\frac{1}{n_2}-\frac{1}{n_3}\Big),\quad \beta=\frac{1}{2}\Big(1-\frac{1}{n_1}-\frac{1}{n_2}+\frac{1}{n_3}\Big),\quad \gamma=1-\frac{1}{n_1}
\end{equation}
and 
\begin{equation}\label{a'b'c'}
  \alpha'=\frac{1}{2}\Big(1+\frac{1}{n_1}-\frac{1}{n_2}-\frac{1}{n_3}\Big),\quad \beta'=\frac{1}{2}\Big(1+\frac{1}{n_1}-\frac{1}{n_2}+\frac{1}{n_3}\Big),\quad \gamma'=1+\frac{1}{n_1}.  
\end{equation}
Thus,
\begin{align*}
   \frac{1}{n_1}+\frac{1}{n_2}+\frac{1}{n_3}=1-2\alpha,\\
   \frac{1}{n_1}+\frac{1}{n_2}-\frac{1}{n_3}=1-2\beta,\\
   -\frac{1}{n_1}+\frac{1}{n_2}+\frac{1}{n_3}=1-2\alpha',\\
   -\frac{1}{n_1}+\frac{1}{n_2}-\frac{1}{n_3}=1-2\beta'.
\end{align*}
From (\ref{tanh1}), we have
\begin{align*}
    \tanh^2{\frac{d_h(w_1,w_2)}{2}}&=\displaystyle\frac{\cos{\big\{\frac{1}{2}}\big(1-2\alpha\big)\pi\big\}\cos{\big\{\frac{1}{2}}\big(1-2\beta\big)\pi\big\}}{\cos{\big\{\frac{1}{2}}\big(1-2\alpha'\big)\pi\big\}\cos{\big\{-\frac{1}{2}}\big(1-2\beta'\big)\pi\big\}}\\
    &=\displaystyle\frac{\sin{\pi \alpha}\sin{\pi \beta}}{\sin{\pi \alpha'}\sin{\pi \beta'}}.
\end{align*}
\end{proof} 

Recall that the linearly independent solutions of (\ref{hyp}) are $g_1(\xi)={_2F_1}(\alpha, \beta; \gamma;\xi)$ and $g_2(\xi)=\xi^{1-\gamma} {_2F_1(\alpha', \beta'; \gamma';\xi)}$ for $0<\gamma<1$. The quotient $w=f(\xi)=\displaystyle \frac{g_2(\xi)}{g_1(\xi)}$ maps conformally $\mathcal{H}$ to the interior of $[w_1, w_2, w_3]$. Let $\partial[w_1, w_2, w_3]$ and $\partial\mathcal{H}$ denote the boundary of the triangle $[w_1, w_2, w_3]$ and the boundary of $\mathcal{H}$, respectively. Then $f(\xi)$ forms a homeomorphism between $\partial\mathcal{H}=\mathbb{R}\cup \{\infty\}$ and $\partial[w_1, w_2, w_3]$. The orders of $w_1,\, w_2$ and $w_3$ are $n_1,\,n_2$ and $n_3$, respectively, where $n_1,\,n_2,\,n_3$ are integers greater than $2$ or infinity. Since \[\frac{1}{n_1}+\frac{1}{n_2}+\frac{1}{n_3}<1,\] the inverse $\xi(w)$ of the Schwarz triangle function $f(\xi)$ is single-valued (see p. 416 in \cite{sansone}). Also, $\xi(w_1)=0,\,\xi(w_2)=1$ and $\xi(w_3)=\infty$. 

Consider the elliptic point $w_1\in\mathcal{H}$ of order $n_1$ and let $A\in G$ be the generator of the stabilizer subgroup for $w_1$, then 

\begin{equation}\label{At}
    \frac{A w-w_1}{Aw-\overline{w}_1}=e^{\frac{2\pi i}{n_1}}\frac{w-w_1}{w-\overline{w}_1}
\end{equation}
(see (9) of Section 48 in \cite{ford}). Also, for $\begin{pmatrix}  a & b\\ c & d \end{pmatrix}\in G$ and $w=\frac{g_2}{g_1}$, we have 
\begin{equation}\label{tau}
    w\mapsto\begin{pmatrix}  a & b\\ c & d \end{pmatrix}w=\frac{a w+ b}{c w+ d}=\frac{a g_2+b g_1}{c g_2+ d g_1}.
\end{equation}
The two sides (hyperbolic lines) of $[w_1, w_2, w_3]$ through $w_1$ are mapped to straight lines through the origin by the M\"{o}bius transformation $$w\rightarrow \displaystyle \frac{w-w_1}{w-\overline{w}_1}.$$

\begin{myproof}{Theorem}{\ref{thm1}}
The points on $G\backslash\mathcal{H}$ are the $G$-orbits of $w$ near the elliptic point $w_1$ such that 
\begin{equation}\label{k}
    \frac{w-w_1}{w-\overline{w}_1}=K\frac{w_2}{w_1}=K\frac{\xi^{1-\gamma} {_2F_1(\alpha', \beta'; \gamma';\xi)}}{{_2F_1(\alpha, \beta; \gamma;\xi)}}
\end{equation}
for a nonzero constant $K$ (see (48) of \cite{elki}).
By assumptions, $\xi(w_2)=1$ and 
\begin{align*}
    \gamma'-\alpha'-\beta'&=\gamma-\alpha-\beta,\\
    \gamma'-\beta'&=1-\beta,\\
    \gamma'-\alpha'&=1-\alpha.
\end{align*}
Thus,
\begin{align*}
    {_2F_1(\alpha, \beta; \gamma;\xi(w_2))}=\frac{\Gamma(\gamma)\Gamma(\gamma-\alpha-\beta)}{\Gamma(\gamma-\alpha) \Gamma(\gamma-\beta)}
\end{align*}
(see Section 14.11 of \cite{bateman}) and
\begin{align*}
        {_2F_1(\alpha', \beta'; \gamma';\xi(w_2))}&=\frac{\Gamma(\gamma')\Gamma(\gamma'-\alpha'-\beta')}{\Gamma(\gamma'-\alpha') \Gamma(\gamma'-\beta')}\\
        &=\frac{\Gamma(\gamma')\Gamma(\gamma-\alpha-\beta)}{\Gamma(1-\alpha) \Gamma(1-\beta)}.
\end{align*}
Therefore, when $w \rightarrow w_2$, from (\ref{k}) we have
\begin{equation*}
   \displaystyle \frac{w_2-w_1}{w_2-\overline{w}_1} =K\frac{\Gamma(\gamma')\Gamma(\gamma-\alpha)\Gamma(\gamma-\beta)}{\Gamma(\gamma)\Gamma(1-\alpha)\Gamma(1-\beta)}
\end{equation*}
or,
\begin{equation}{\label{K}}
    K=\displaystyle \frac{w_2-w_1}{w_2-\overline{w}_1}\frac{\Gamma(\gamma)\Gamma(1-\alpha)\Gamma(1-\beta)}{\Gamma(\gamma')\Gamma(1-\alpha')\Gamma(1-\beta')}.
\end{equation}
Since $\sin{\pi \alpha=\displaystyle\frac{\pi}{\Gamma(\alpha)\Gamma(1-\alpha)}}$ for $\alpha\notin \mathbb{Z}$, by Lemma \ref{lem1} we have 
\begin{equation}{\label{tanh3}}
    \tanh^2{\frac{d_h(w_1,w_2)}{2}}
    =\displaystyle\frac{\sin{\pi \alpha}\sin{\pi \beta}}{\sin{\pi \alpha'}\sin{\pi \beta'}}=\displaystyle\frac{\Gamma(\alpha')\Gamma(1-\alpha')\Gamma(\beta')\Gamma(1-\beta')}{\Gamma(\alpha)\Gamma(1-\alpha)\Gamma(\beta)\Gamma(1-\beta)}.
\end{equation}
From (\ref{tanh2}), it follows that
\begin{equation}{\label{tanh4}}
    \tanh^2{\frac{d_h(w_1,w_2)}{2}}=\Big(\displaystyle \frac{w_2-w_1}{w_2-\overline{w}_1}\Big)^2.
\end{equation}
Thus, from (\ref{K}), (\ref{tanh3}) and (\ref{tanh4}), we have 
\begin{align*}
    K=\bigg(\displaystyle\frac{\Gamma(1-\alpha)\Gamma(1-\beta)\Gamma(\alpha')\Gamma(\beta')}{\Gamma(\alpha)\Gamma(\beta)\Gamma(1-\alpha')\Gamma(1-\beta')}\bigg)^\frac{1}{2}\frac{\Gamma(\gamma)}{\Gamma(\gamma')}.
\end{align*}

Now, (\ref{k}) implies 
\begin{equation}\label{t}
    w=\frac{w_1 g_1-K\overline{w}_1g_2}{g_1-Kg_2}.
\end{equation}
Therefore,
\begin{equation}\label{dt}
    \frac{dw}{d\xi}=K(w_1-\overline{w}_1) \frac{g_1\displaystyle\frac{dg_2}{d\xi}-\displaystyle\frac{dg_1}{d\xi}g_2}{(g_1-Kg_2)^2}.
\end{equation}
Let $W$ denote the Wronskian of $g_1$ and $g_2$ given by $$W=g_1\displaystyle\frac{dg_2}{d\xi}-\displaystyle\frac{dg_1}{d\xi}g_2.$$ 
Then, 
\begin{equation}\label{dt2}
    \frac{dw}{d\xi}=K(w_1-\overline{w}_1) \frac{W}{(g_1-Kg_2)^2}.
\end{equation}
From (\ref{hyp}), we have 
\begin{equation}\label{fu}
    \frac{d^2g}{d\xi^2}+P(\xi)\frac{dg}{d\xi}+Q(\xi)g=0,
\end{equation}
where \[P(\xi)=\displaystyle \frac{\gamma-(\alpha+\beta+1)\xi}{\xi(1-\xi)}\]and  \[Q(\xi)=\displaystyle \frac{\alpha\beta}{\xi(\xi-1)}.\] Since the linearly independent solutions of (\ref{fu}) are $g_1(\xi)$ and $g_2(\xi)$, we have by Abel's theorem (see Theorem 3.3.2 of \cite{boyce})
\begin{equation}\label{w1}
    W(\xi)=W(\xi_0) \exp{\Big(-\int\limits_{\xi_0}^\xi}P(\xi) d\xi\Big),
\end{equation}
where $W(\xi_0)\neq0$. For some constant $C_0\in \mathbb{C}\setminus \{0\}$, it follows from (\ref{w1}) that 
\begin{equation}\label{w2}
    W(\xi)=C_0 \xi^{-\gamma}(1-\xi)^{\gamma-\alpha-\beta-1}.
\end{equation}
Therefore, from (\ref{dt2}) and (\ref{w2}) we have  

\begin{equation}
    \frac{dw}{d\xi}=C_0K(w_1-\overline{w}_1) \frac{\xi^{-\gamma}(1-\xi)^{\gamma-\alpha-\beta-1}}{(g_1-Kg_2)^2},
\end{equation}
which implies
\begin{equation*}
    \displaystyle \frac{d\xi(w)}{dw}=C \xi^{(1-\frac{1}{n_1})}(1-\xi)^{(1-\frac{1}{n_2})}\Big({_2F_1}(\alpha, \beta; \gamma;\xi)-K \xi^{\frac{1}{n_1}}\ _2F_1(\alpha', \beta'; \gamma';\xi)\Big)^2,
\end{equation*}
where 
\begin{equation}\label{C}
    C=\displaystyle \frac{1}{C_0K(w_1-\overline{w}_1)}.
\end{equation}
\end{myproof}

\begin{corollary}
    The constants $C$ and $K$ in Theorem \ref{thm1} are inversely proportional to each other. 
\end{corollary}
\begin{proof}
    From (\ref{C}), it follows that $C\propto\frac{1}{K}$.
\end{proof}

\def\cprime{$'$} \def\cprime{$'$} \def\cprime{$'$}
\providecommand{\bysame}{\leavevmode\hbox to3em{\hrulefill}\thinspace}
\providecommand{\MR}{\relax\ifhmode\unskip\space\fi MR }

\providecommand{\MRhref}[2]{%
  \href{http://www.ams.org/mathscinet-getitem?mr=#1}{#2}
}
\providecommand{\href}[2]{#2}

\end{document}